\documentclass[a4paper,12pt,reqno]{amsart}
\usepackage{amsmath}
\usepackage{amsthm}
\usepackage{amsfonts}
\usepackage{amssymb}
\usepackage{amscd}
\usepackage{bm}
\usepackage{color}
\usepackage{draftcopy}
\usepackage{exscale}
\usepackage{extarrows}
\usepackage{epstopdf}
\usepackage{epsfig}
\usepackage{enumerate}
\usepackage{graphics}
\usepackage{graphicx}
\usepackage{latexsym}
\usepackage{mathrsfs}
\usepackage{pifont}
\usepackage{paralist}
\usepackage{esint}
\usepackage{boxedminipage}
\renewcommand\eqref[1]{(\ref{#1})} 
\usepackage[toc,page]{appendix}

\usepackage{hyperref}
\renewcommand\eqref[1]{(\ref{#1})} 

\usepackage[a4paper, hmargin={2.7cm,2.7cm},vmargin={3.3cm,3.3cm}]{geometry}

\newtheorem{theorem}{Theorem}[section]
\newtheorem{definition}{Definition}[section]
\newtheorem{proposition}[theorem]{Proposition}
\newtheorem{lemma}[theorem]{Lemma}

\newtheorem{remark}[theorem]{Remark}

\numberwithin{equation}{section}

\newcommand{\esssup}{\mathop{\mathrm{ess\,sup}}}

\newcommand{\Ree}{\operatorname{Re}}

\newcommand{\x}{\mathbf{x}}
\newcommand{\y}{\mathbf{y}}
\newcommand{\xx}{\mathrm{x}}
\newcommand{\yy}{\mathrm{y}}

\begin{document}

\title[Lipschitz spaces adapted to Schr\"{o}dinger operators]{Lipschitz spaces  adapted to Schr\"{o}dinger operators on the Heisenberg group}

\author[Qing Hong]{Qing Hong${ }^1$}
\thanks{$^1$School of Mathematics and Statistics,
 Jiangxi Normal University,
Nanchang, Jiangxi 330022, China}

\author[Xiuzhen Hou]{Xiuzhen Hou${ }^1$}

\author[Guorong Hu]{Guorong Hu${}^{1 \ast}$}
\thanks{$^\ast$Corresponding author, Email: hugr@mail.ustc.edu.cn}






\thanks{Q. Hong is supported by the National Natural Science Foundation of China (Grant No. 12361017) and 
the Natural Science Foundation of Jiangxi Province (Grant No. 20242BAB25002). G. Hu is supported by the  National Natural Science Foundation of China (Grant No. 12461018)}

\begin{abstract}
Let $L =-\Delta_{\mathbb{H}^n} +V$ be the Sch\"{o}dinger operator on the
Heisenberg group $\mathbb{H}^n$, where $\Delta_{\mathbb{H}^n}$ is the sub-Laplacian, and $V$ is a nonnegative potential belonging to the reverse H\"{o}lder class $RH_q(\mathbb{H}^n)$ for some $q > Q/2$, where $Q:=2n+2$ is the homogeneous dimension of $\mathbb{H}^n$. 
In this paper, motivated by the work of De Le\'{o}n-Contreras and Torrea \cite{DT}, we introduce the Lipschitz spaces $\Lambda_L^\alpha (\mathbb{H}^n)$,   $0< \alpha <2$,  adapted to $L$ via 
a pointwise second-order difference condition involving the critical radius function $\rho$ related to $V$, and
also introduce another type of Lipschitz spaces 
$\Gamma^{\alpha/2}_L(\mathbb{H}^n)$, $0< \alpha <\infty$, adapted to $L$ in terms of the heat semigroup $e^{-tL}$. 
We show that for $0< \alpha <2-(Q/q)$, $\Lambda_{L}^\alpha (\mathbb{H}^n) =\Gamma_L^{\alpha/2} (\mathbb{H}^n)$ with equivalent norms. Applications of $\Gamma^{\alpha/2}_L(\mathbb{H}^n)$ to the regularity of the fractional powers of the operator $L$ are also given.
\end{abstract}
 
\maketitle

\tableofcontents

\section{Introduction and statement of main result}
\allowdisplaybreaks
\subsection{Motivation}
For any $\alpha >0$,  the classical Lipschiz (or H\"{o}lder-Zygmund) space 
$\Lambda^\alpha (\mathbb{R}^n)$ is defined as the set of all measurable functions $f$ on $\mathbb{R}^n$
such that
\begin{align*}
\|f\|_{\Lambda^\alpha(\mathbb{R}^n)}: = \|f\|_{L^\infty (\mathbb{R}^n)} + \sup_{h \in \mathbb{R}^n \backslash \{0\}}
 \frac{ \|D_h^{\lfloor \alpha \rfloor +1} f\|_{L^\infty(\mathbb{R}^n)}}{|h|^\alpha} <\infty,
\end{align*}
where $\lfloor \alpha \rfloor:= \max \{\ell \in \mathbb{Z}: \ell \leq \alpha\}$.
Here, for any $k \in \mathbb{N}$,  $D^{k}_h f$ denotes the $k$-th order difference of $f$ which is defined 
iteratively as follows:
\begin{align*}
 D_h^1f(x) := f(x+h) -f(x), \quad
 D^{k}_h f(x) :=D_h^1 ( D_h^{k-1}f)(x), \quad k  \geq 2.
\end{align*}
This class of functions  play an important role in harmonic analysis and partial differential equations.

Let $L =-\Delta + V$ be the Schr\"{o}dinger operator
on $\mathbb{R}^n$, $n \geq 3$, where $V$ is a nonnegative potential belonging to the reverse H\"{o}lder class $RH_q (\mathbb{R}^n)$.  Bongioanni, Harboure and Salinas \cite{BHS-weighted} first introduced Lipschitz classes $\Lambda^\alpha_L(\mathbb{R}^n)$, $0< \alpha <1$, adapted to  the operator $L$, 
via the norms
\[
 \|f\|_{\Lambda^\alpha_L(\mathbb{R}^n)}:= \|\rho(\cdot)^{-\alpha}f(\cdot)\|_{L^\infty(\mathbb{R}^n)} +  \sup_{h \in \mathbb{R}^n \backslash \{0\}}
  \frac{\|f(\cdot + h) - f(\cdot)\|_{L^\infty}}{|z|^\alpha},
\]
and  showed that this space coincides with the space $BMO_L^\alpha(\mathbb{R}^n)$ which  is 
 defined as the space of locally integrable functions such that
 \[
 \int_B \left|f(x)-\frac{1}{|B|}\int_B f(y)dy\right|dx \leq C |B|^{1 +\frac{\alpha}{n}} \quad \text{for every ball } B =B(x_B,r_B),
 \]
 and
 \[
 \int_B |f(x)|dx \leq C |B|^{1 +\frac{\alpha}{n}} \quad \text{if } r_B \geq \rho(x_B).
 \]
They also study the behavior of the fractional integral $L^{-\beta/2}$ on   $\Lambda_L^\alpha(\mathbb{R}^n)$ for $0< \alpha <1$, $\beta \geq 0$
and $\alpha + \beta < \min\{1,2 -(n/q)\}$.
 Ma, Stinga, Torrea and Zhang in \cite{MSTZ} prove a characterization of the space $\Lambda_L^\alpha(\mathbb{R}^n)$, for $0<\alpha <1$, in
 terms of Carleson measures and fractional derivatives of any order of the Poisson semigroup $e^{-t\sqrt{L}}$, and studied the behavior of the
 fractional powers $L^{\beta/2}$ (positive and negative) and  multipliers of Laplace transform type on the space $\Lambda_L^\alpha(\mathbb{R}^n)$.
De Le\'{o}n-Contreras and Torrea \cite{DT}  extended the 
definition of $\Lambda_L^\alpha(\mathbb{R}^n)$ to $0< \alpha <2$ via second difference (see Definition~\ref{def:Lipschitz}), and 
generalized the regularity results obtained in \cite{MSTZ}.
In the case of the Hermite operator $H =-\Delta + |x|^2$, the associated H\"{o}lder spaces $C^{k,\alpha}_H(\mathbb{R}^n)$, for $k \in \mathbb{N}_0$ and $0< \alpha \leq 1$,  
 were investigated by Stinga and Torrea \cite{ST1}.  
 Recently, Hong, Xu and Hu \cite{HXH} proved the Littlewood-Paley and Carleson
 measure characterization of $\Lambda^\alpha_L(\mathbb{R}^n)$, for $0< \alpha < 2-n/q$.

The study of Lipschitz spaces on the Heisenberg group and more general stratified Lie groups has a long history. Lipschitz spaces $\Lambda^\alpha (G)$, $0< \alpha <\infty$, on an arbitrary stratified Lie group $G$ were first
introduced by Folland \cite{Folland}. Some characterizations of $\Lambda^\alpha(G)$ were given by Folland \cite{Folland-2}, Krantz \cite{Krantz} and Hu \cite{Hu-CMJ}.
Recently, the authors in \cite{HHLT-1} studied Marcinkiewicz multipliers and their boundedness on Lipschitz spaces on the Heisenberg group, while the authors in \cite{HHLT-2} studied boundedness of multi-parameter singular integrals on H\"{o}lder spaces associated with Zygmund dilations. However, all these works are restricted to the case where the underlying operator is the sub-Laplacian or its variants, without considering the presence of a nonnegative potential $V$.

More recently,  Dai \cite{Dai} introduced Campanato-type spaces $BMO_L^\alpha (\mathbb{H}^n)$  
adapted to the Schr\"{o}dinger operator $L =-\Delta_{\mathbb{H}^n} +V$
on the Heisenberg group for $0< \alpha \leq 1$, and established a 
$T1$ criterion for the boundedness of 
$\gamma$-Schr\"{o}dinger-Calder\'{o}n-Zygmund operators on these spaces. 
As applications, the author in \cite{Dai} proved the boundedness of maximal operators, Littlewood-Paley 
$g$-functions, Laplace transform type multipliers, and fractional integrals associated with 
$L$ on $BMO_L^\alpha (\mathbb{H}^n)$. Moreover,  it is shown in \cite{Dai} that
for $0<\alpha \leq 1$, $BMO_L^\alpha (\mathbb{H}^n)$ coincides with the H\"{o}lder space
$C^{0,\alpha}_L(\mathbb{H}^n)$ adapted to  the operator $L$. It is worth pointing out that for $0< \alpha <1$,
the H\"{o}lder space $C^{0,\alpha}_L (\mathbb{H}^n)$ coincides with the Lipschitz space $\Lambda^\alpha_L (\mathbb{H}^n)$;
however, for $\alpha =1$, such coincidence fails. 
We also observe that heat semigroup characterization and the
regularity of fractional powers of $L$ have not been investigated in  \cite{Dai}.

In the present paper, inspired by the work of De Le\'{o}n-Contreras and Torrea \cite{DT},
we will introduce two types of Lipschitz spaces adapted to the  
Schr\"{o}dinger operator $L =-\Delta_{\mathbb{H}^n} +V$ on the Heisenberg group. 
The first type of Lipschitz spaces, $\Lambda_L^\alpha (\mathbb{H}^n)$,  with $0< \alpha <2$, are defined via 
a pointwise second-order difference condition involving the critical radius function $\rho$ related to $V$. 
The second type of Lipschitz spaces, 
$\Gamma^\alpha_L(\mathbb{H}^n)$, with $0< \alpha <\infty$,  are introduced in terms of the heat semigroup $e^{-tL}$. 
We will study the relation between $\Lambda_L^\alpha (\mathbb{H}^n)$ and $\Gamma_L^\alpha (\mathbb{H}^n)$, and
give some applications of these spaces to the regularity of  the fractional powers of the operator $L$.

\subsection{The Heisenberg group $\mathbb{H}^n$}

The Heisenberg group $\mathbb{H}^n$ is the Lie group with underlying manifold $\mathbb{R}^{2n}\times \mathbb{R}$ and multiplication.
\begin{align*}
\mathbf{x} \mathbf{y} =  \bigg(\xx_1+ \yy_1, \  \dots, \  \xx_{2n} + \yy_{2n}, \
\xx_{2n+1} +\yy_{2n+1} + 2 \sum_{j=1}^n \big(\xx_{n+j}\yy_j-\xx_j \yy_{n+j}\big)\bigg)
\end{align*} 
for all points $\mathbf{x}=(\xx_1, \dots, \xx_{2n+1})$  and $\mathbf{y} = (\yy_1,\dots, \yy_{2n +1})$ in $\mathbb{H}^n$.
It is well known that for any $\x = (\xx_1,\dots, \xx_{2n+1})\in \mathbb{H}^n$, the inverse of $\x$ coincides with $-\x
=(-\yy_1,\dots, \yy_{2n+1})$, i.e., $\x^{-1} = -\x$. 
We denote the group identity of $\mathbb{H}^n$ by $\mathbf{0} := (0,\dots, 0)$

Let  $\mathfrak{h}_n$ be the Lie algebra of left-invariant vector fields on $\mathbb{H}^n$.
A basis of $\mathfrak{h}_n$ is the following:
{\small \begin{align*}
X_{2n+1} = \frac{\partial}{\partial \xx_{2n+1}}, \quad  X_j= \frac{\partial }{\partial \xx_j} +2\xx_{n+j} \frac{\partial }{\partial \xx_{2n+1}},  \quad  X_{n+j}  =\frac{\partial }{\partial \xx_{n+j}} -2\xx_j \frac{\partial }{ \partial \xx_{2n+1}}, \quad
j =1,\dots, n.
\end{align*}
All} non-trivial commutation relations are given by
\begin{align*}
 [X_j, X_{n+j}] = -4X_{2n+1},\quad j =1,\dots, n.   
\end{align*}
The  sub-Laplacian on $\mathbb{H}^n$ are defined by
\begin{align*}
\Delta_{\mathbb{H}^n}: = \sum_{j =1}^{2n} X_j^2.
\end{align*}

For any $r>0$, the dilation $\delta_r$ on $\mathbb{H}^n$ is defined by
\begin{align*}
\delta_r (\mathbf{x})  = (r \xx_1, \dots, r\xx_{2n},r^2 \xx_{2n+1}), \quad \x =  (\xx_1, \dots, \xx_{2n+1}) \in \mathbb{H}^n.
\end{align*}
The {\it Kor\'{a}nyi norm} on $\mathbb{H}^n$ is given by
\begin{align*}
|\mathbf{x}| = \left(\left(\sum_{j=1}^{2n} \xx_j^2\right)^2 +16 \xx_{2n+1}^2\right)^{1/4}, \quad \mathbf{x}   = (\xx_1, \dots, \xx_{2n+1}) \in \mathbb{H}^n.
\end{align*}
This norm is homogeneous of degree $1$ with respect to the dilations, that is, $|\delta_r(\x)| = r|\x|$ for all $\x \in
\mathbb{H}^n$ and $r >0$. Moreover, $|\cdot|$ satisfies the triangle inequality (see \cite{Cygan})
\begin{align}
|\x \y | \leq |\x| +|\y|    
\end{align}
for all $\x, \y \in \mathbb{H}^n$, 
and hence leads to a left-invariant distance $d(\x,\y) =|\y^{-1}\x|$.
The ball with radius $r$ centered at $\x$ is defined by
\begin{align*}
B(\x,r) := \{\y\in \mathbb{H}^n : d(\x,\y) <r\}.
\end{align*}

It is well known that the Haar measure $\mu$ on $\mathbb{H}^n$ coincides with the Lebesgue measure on $\mathbb{R}^
{2n}\times \mathbb{R}$.
We denote by $|E|$ the Haar measure of any measurable set $E \subset \mathbb{H}^n$, i.e.,
$|E|:=\mu(E)$. It is well-known that
\begin{align*}
|\delta_r (E)| = r^Q |E|,
\end{align*}
where $Q:= 2n+2$ is the homogeneous dimension of $\mathbb{H}^n$. Particularly,
the volume (Haar measure) of a ball $B(\x,r)$ is
\begin{align*}
|B(\x,r)| = |B(\mathbf{0},1)|r^Q =\frac{2\pi^{n+\frac{1}{2}}\Gamma(\frac{n}{2})}{(n+1)\Gamma(n)\Gamma(\frac{n+1}{2})}r^{Q}.
\end{align*}

Integrations on (measurable subsets of) $\mathbb{H}^n$ are with respect to the Haar measure $\mu$. 
Using an analogue of polar coordinates on homogeneous groups, cf. \cite[Proposition 1.15]{FS}, 
it is easy to see that
\begin{align*}
\int_{\mathbb{H}^n} (1 + |\x|)^{-(n +a)} d\mu(\x)<\infty \Longleftrightarrow   a >0.
\end{align*}
The convolution of two functions $f$ and $g$ on $\mathbb{H}^n$ are defined by
\begin{align*}
f \ast g (\x) = \int_{\mathbb{H}^n} f(\y)g(\y^{-1}\x)d\mu(\x)= \int_{\mathbb{H}^n} f(\x \y^{-1})g(\y)d\mu(\y),
\end{align*}
provided that the integrals makes sense.

\subsection{Lipschitz spaces adapted to  $L =-\Delta_{\mathbb{H}^n}+V$}

Consider the  Schr\"{o}dinger operator 
\[
L =- \Delta_{\mathbb{H}^n} + V
\]
 on the Heisenberg group $\mathbb{H}^n$. Throughout this paper, 
we assume that the potential $V$ is nonnegative and belongs to the reverse H\"{o}lder class $RH_q$ for some $q > Q/2$, that is, there exists a constant $C >0$
such that for every ball $B \subset \mathbb{H}^n$,
\begin{align} \label{eq:RHC}
\left( \frac{1}{|B|} \int_B V(\x)^q d\mu(\x)\right)^{1/q} \leq C \frac{1}{|B|} \int_B V(\x)d\mu(\x).
\end{align}

Following \cite{ShenAIF}, we define the critical radius function $\rho(\cdot)$ related to $L$ by
\begin{align} \label{eq:triangle}
 \rho(\x):=\sup\left\{ r>0 :\frac{r^2}{|B(\x, r)|} \int_{B(\x, r)} V(\y) d\mu(\y)
 \leq 1 \right\},\quad  \x \in \mathbb{H}^n.
\end{align}
Similarly as in the Euclidean case (see \cite[p. 515]{ShenAIF}), 
the reverse H\"{o}lder condition \eqref{eq:RHC} implies that
\[
\lim_{r \rightarrow 0^+} \frac{r^2}{|B(\x, r)|} \int_{B(\x, r)} V(\y) d\mu(\y) =0.
\]
Consequently $\rho (\x) >0$ for every $\x \in \mathbb{H}^n$.
We will show that (see Lemma \ref{lem:rho finite}) if
\begin{align} \label{eq:not vanishing}
\big|\{\x \in \mathbb{H}^n :  V(\x) >0\}\big| \neq 0,
\end{align}
then $0< \rho(\x) <\infty$ for every $\x \in \mathbb{H}^n$.
We will require that $\rho <\infty$ in many places in our argument. 
 Hence, we always assume that \eqref{eq:not vanishing} holds.

As a natural generalizaton of the Lipschitz spaces adapted to Schr\"{o}dinger
operators on the Euclidean spaces, we introduce the Lipschitz spaces adapted to 
Schr\"{o}dinger operators on the Heisenberg group $\mathbb{H}^n$.

\begin{definition} \label{def:Lipschitz} 
Let $0 < \alpha <2$. The Lipschitz space $\Lambda^\alpha_{L} (\mathbb{H}^n)$ adapted to $L$ is defined as 
the set of all measurable functions $f$ on $\mathbb{H}^n$ such that
\begin{align*}
 \|f\|_{\Lambda^\alpha_{L}(\mathbb{H}^n)}:= \|\rho(\cdot)^{-\alpha}f\|_{L^\infty(\mathbb{H}^n)} + 
 \sup_{|\y|>0}
  \frac{\|f(\x \y) + f(\x \y^{-1}) -2 f(\x)\|_{L^\infty(\mathbb{H}^n_\x)}}{|\y|^\alpha} <\infty,
\end{align*}
where $\|\cdot\|_{L^\infty(\mathbb{H}^n_\x)}$ means an $L^\infty$ norm with respect to the variable $\x$, i.e.,
\begin{align*}
 \|f(\x \y) + f(\x \y^{-1}) -2 f(\x)\|_{L^\infty(\mathbb{H}^n_\x)}:= \esssup_{x \in \mathbb{H}^n}
 |f(\x \y) + f(\x \y^{-1}) -2 f(\x)|.
\end{align*}
\end{definition}

We shall also introduce Lipschitz spaces in terms of the heat semigroup associated with $L$. 
For this purpose, we need to introduce 
some notions and notation. Let
\begin{align} \label{heat semigroups}
 T_t:=e^{t\Delta_{\mathbb{H}^n}} \quad \text{and} \quad T^L_t:=e^{-t L}
\end{align}
be the heat semigroups associted to $-\Delta_{\mathbb{H}^n}$ and $L =- \Delta_{\mathbb{H}^n} + V$, respectively.
Following \cite{DT}, we say that a function $f$ satisfies the {\bf heat size condition for $L$}, if $f$ is a measurable
function on $\mathbb{H}^n$ and satisfies the following properties:
\begin{itemize}
     \item for every $t>0$, $\int_{\mathbb{H}^n}e^{- \frac{|\x|^2}{t}}|f(\x)|d\mu(\x)<\infty$; 
     \item for every $t>0$, every  $\ell\in\mathbb{N}_0$, and every $\x\in\mathbb{H}^n$, 
     $\lim_{t \to \infty}\partial_t^\ell T_t^L f(\x)=0$.
     \end{itemize}  
    We denote
\begin{align*}
  \mathcal{F}_L (\mathbb{H}^n) := \left\{f : f \text{ satisfies the heat size condition for } L\right\}.  
\end{align*}

\begin{definition}
Let $\alpha >0$ and $k:= \lfloor \frac{\alpha}{2} \rfloor +1$.
The Lipschitz space $\Gamma_L^{\alpha/2} (\mathbb{H}^n)$ adapted 
to $L$ is defined as the set of all $f \in \mathcal{F}_L(\mathbb{H}^n)$ such that 
\begin{align*}
 \|f\|_{\Gamma_{L}^{\alpha/2} (\mathbb{H}^n)} := \sup_{t >0}t^{k-\frac{\alpha}{2}} \big\|\partial_t ^k T_t^L f \big\|_{L^\infty (\mathbb{H}^n)} <\infty .
\end{align*}
\end{definition}

\subsection{Main result and novelty of the paper}
The main result of the present paper is the equivalence between the pointwise and heat semigroup definitions
of the Lipschitz spaces adated to $L =-\Delta_{\mathbb{H}^n} +V$, for a suitable range of $\alpha$.

\begin{theorem} \label{thm:main-1}
Let $0 <\alpha <2 - \frac{Q}{q}$. Then 
\begin{align*}
 \Lambda_L^\alpha (\mathbb{H}^n) =\Gamma_{L}^{\alpha /2}(\mathbb{H}^n)
\end{align*}
with equivalence of their norms.
\end{theorem}

To prove Theorem \ref{thm:main-1}, we will adopt the idea from \cite{DT}, where such a characterization 
was established in the Euclidean setting. 
However, unlike the Euclidean setting, the Heisenberg group 
is a non-commutative, stratified nilpotent Lie group with a sub-Riemannian structure. 
This lack of commutativity and the presence of a non-isotropic dilation structure introduce substantial new difficulties in the characterization of Lipschitz spaces adapted to Schr\"{o}dinger operators. 
For example, the heat kernel estimates for the sub-Laplacian $\Delta_{\mathbb{H}^n}$ 
 involve the homogeneous dimension 
$Q=2n+2$ and exhibit anisotropic decay in the vertical direction. As a consequence, 
the second-order difference condition in the definition of $\Lambda^\alpha_L (\mathbb{H}^n)$ 
 involves products of the form $f(\x\y) +f(\x\y^{-1}) -2f(\x)$,
which cannot be reduced to Euclidean translations. This requires a refined Taylor formula on the group (see Lemma 2.7) that uses the left-invariant vector fields and a careful control of the derivatives of the heat semigroup along 
these non-commuting directions.
 
\medskip
The rest of this paper is organized as follows. In Section \ref{sec:preliminaries}, we give some preliminaries,
including properties of the critical radius function $\rho$, some estimates of the heat kernels of the operators 
$-\Delta_{\mathbb{H}^n}$ and $L =-\Delta_{\mathbb{H}^n} +V$, and a finite increment formula on $\mathbb{H}^n$.
Section \ref{sec:proof of main} is devoted to the proof of our main result, Theorem \ref{thm:main-1}.
Some applications of the Lipschitz spaces adapted to $L$, including the behavior of the Bessel potentials, fractional integrals and fractional ``sub-Laplacian''
associated to the operator $L =-\Delta_{\mathbb{H}^n}+V$ on the Lipschitz spaces $\Gamma_L^\alpha(\mathbb{R}^n)$, 
will be given in Section \ref{sec:regularity}.

\medskip
\noindent
{\it Notation.} The set of all positive integers is denoted by $\mathbb{N}$,
while the set of all nonnegative integers is denoted by $\mathbb{N}_0$. If $a \in (0,\infty)$, we denote 
$\lfloor a \rfloor : =\max\{k \in \mathbb{N}_0: k \leq a\}$. 
Throughout, we will use $c,c',C,C'$ to denote positive constants, which are independent of the main
variables involved and whose values may vary at every occurrence.
By writing $f \lesssim g$, we mean $f \leq Cg$.
The notation $f \asymp g$ will stand for $C \leq f/g \leq C'$.





\section{Preliminaries} \label{sec:preliminaries}
\subsection{The critical  radius function}
The following property of the critical radius function $\rho$ was proved in \cite{Lu}.
\begin{lemma}\label{lem:property of rho}
There exist constants $C > 0$ and $k_0 \geq 1$ such that for all $\x, \y \in \mathbb{H}^n$, 
\begin{align*}
 C^{-1} \rho(\x) \left(1 + \frac{|\y^{-1}\x|}{\rho(\x)}\right)^{-k_0} \le \rho(\y) \le C \rho(\x) \left(1 + \frac{|\y^{-1}\x|}{\rho(\x)}\right)^{\frac{k_0}{k_0+1}}.
 \end{align*}
\end{lemma}

\begin{remark}
From Lemma \ref{lem:property of rho} it is easy to see that $\rho(\x_0) < \infty$ for some $\x_0 \in \mathbb{H}^n$ 
if and only if $\rho(\x)  <\infty$ for every $\x \in \mathbb{H}^n$.
\end{remark}

Using Lemma \ref{lem:property of rho} and the reverse H\"{o}lder condition \eqref{eq:RHC},
and arguing similarly as the proof of \cite[Lemma 2.2]{HXH}, one  can show the following result.
\begin{lemma} \label{lem:rho finite}
The following two conditions are equivalent:
\begin{enumerate}[\rm (i)]
    \item $\big|\{\x \in \mathbb{H}^n: V(\x) >0\}\big| \neq 0$.
    \item $\rho (\x) <\infty$ for every $\x \in \mathbb{H}^n$.
\end{enumerate}
\end{lemma}

\begin{remark} \label{rmk:growth of rho}
If  $\big|\{\x \in \mathbb{H}^n: V(\x) >0\}\big| \neq 0$, then there exists a constant $C$ such that
\begin{align} \label{eq:rho and 1+x}
  \rho(\x)  \leq C(1 + |\x|)^{\frac{k_0}{k_0 +1}}.
\end{align}
Indeed, since $0< \rho(\mathbf{0}) <\infty$, it follows from Lemma \ref{lem:property of rho} that
\begin{align*}
 \rho(\x) \leq  C\bigg(1 + \frac{|\x|}{\rho(\mathbf{0})}\bigg)^{\frac{k_0}{k_0 +1}}   \leq C \bigg(1 + \frac{1}{\rho(\mathbf{0})}\bigg)^{\frac{k_0}{k_0 +1}} (1 + |\x|)^{\frac{k_0}{k_0 +1}} =C'(1 +|\x|)^{\frac{k_0}{k_0 +1}}.
\end{align*}
\end{remark}

\subsection{Heat kernel  estimates}
It is well known that the heat  semigroup $T_t =e^{t\Delta_{\mathbb{H}^n}}$  consists of
convolution operators. We denote the convolution kernel of $T_t$ by $H_t(\x)$.
\begin{lemma} \label{lem:heat kernel of sub-Laplacian} {\rm  (see \cite[Theorem IV.4.2]{VSCC})}
For any $\ell, m \in \mathbb{N}_0$, and any $(i_1,\dots, i_m) \in \{1,\dots, 2n+1\}^m$, there exist constants
$C,c>0$ such that for all $t >0$ and $\x \in \mathbb{H}^n$,
\begin{align*}
\big|\partial_t^\ell X_{i_1}\cdots X_{i_m} H_t (\x)\big| \leq C t^{-\ell -\frac{Q + \sigma (i_1) +\cdots +\sigma (i_m)}{2}}e^{-\frac{|\x|^2}{ct}},
\end{align*}
 where  $\sigma(i)$, $i =1,\dots ,2n+1$,
denotes the homogeneous degree of   $X_i$, i.e.,
\begin{align} \label{eq:homogeneous degree of X}
  \sigma(i) = \begin{cases}
1, & i =1,\dots, 2n,\\
2, & i =2n+1.
  \end{cases}
\end{align}
\end{lemma}

In general, the heat semigroup $T_t^L =e^{-tL}$ associated to $L =-\Delta_{\mathbb{H}^n} +V$ consists of non-convolution integral operators. We denote the integral kernel of $T_t^L$ by $K_t^L (\x, \y)$. Then we have the following estimate.

\begin{lemma}\label{lem:kernel estimatiaon}
For every $\ell \in \mathbb{N}_0$ and $N >0$, there exist constants $C,c>0$ such that for all $t>0$ and $\x, \y \in \mathbb{H}^n$,
\begin{align} \label{eq:kernel for L}
\big|\partial_t^\ell K_t^L(\x, \y)\big| \le C t^{-\ell -\frac{Q}{2}}e^{-\frac{|\y^{-1}\x|^2}{ct}}
 \bigg(1 + \frac{\sqrt{t}}{\rho(\x)} + \frac{\sqrt{t}}{\rho(\y)}\bigg)^{-N}.
\end{align}
\end{lemma}
 \begin{proof}
It is showed in \cite[Lemma 8]{LinLiu} that 
the semigroup $\{T_t:t>0\}$ extends to a holomorphic semigroup  $\{T_\zeta :\Ree (\zeta) >0\}$ on  
$L^2 (\mathbb{H}^n)$ and its kernel $K_\zeta^L (\x,\y)$ satisfies 
\begin{align*}
|K_\zeta^L(\x,\y)| \leq C_N (\Ree \zeta)^{-\frac{Q}{2}} e^{-\frac{|\y^{-1}\x|^2}{c \Ree (\zeta)}} \left( 1 + \frac{\sqrt{\Ree \zeta}}{\rho(\x)} + \frac{\sqrt{\Ree \zeta}}{\rho(\y )} \right)^{-N}  \quad \text{for } |\arg \zeta| <\frac{\pi}{4}.
\end{align*}
Using Cauchy's integral formula and the above estimate, one can deduce \eqref{eq:kernel for L}. See
\cite[Lemma 2.1]{DT} or \cite[Proposition 3.2]{BHH} for the details.
 \end{proof}

\subsection{A finite increment formula on $\mathbb{H}^n$} 
The following finite increment formula on $\mathbb{H}^n$ will be useful in the proof of our main theorem. It may be regarded as the first-order Taylor formula with integral  remainder on $\mathbb{H}^n$.
\begin{lemma} \label{lem:FTC}
Suppose $f \in C^1(\mathbb{H}^n)$. Then for any two points $\mathbf{x}=(\xx_1,\dots,\xx_{2n+1})$ and $\mathbf{y} = (\mathrm{y}_1,\dots,\mathrm{y}_{2n+1})$ in $\mathbb{H}^n$, we have
\begin{align*}
 f(\mathbf{x}\mathbf{y}) - f(\mathbf{x}) =\int_0^1 \sum_{j =1}^{2n+1} \mathrm{y}_j(X_j f)\big(\x(t\y)
 \big)dt,   
\end{align*}
where $t\y:= (t\yy_1,\dots, t\yy_{2n+1})$.
\end{lemma}

\begin{proof}
We denote by $\mathfrak{h}^n$ the Lie algebra of $\mathbb{H}^n$. Then $X_1,\dots, X_{2n+1}$ is a basis $\mathfrak{h}^n$.
It is well known that (see e.g. \cite[\S 3.1]{BLU}) the  exponential  map $\exp:\mathfrak{h}^n \rightarrow \mathbb{H}^n$ is given by 
\[
\exp\bigg(\sum_{j=1}^{2n+1}\xx_j X_j \bigg) = (\xx_1, \dots,\xx_{2n+1}).
\]
For any point $\y = (\yy_1,\dots, \yy_{2n+1}) \in \mathbb{H}^n$, let $Y:=\exp^{-1}(\y) =\sum_{j=1}^{2n+1}\yy_j X_j\in \mathfrak{h}^n$.
Then define a curve $\gamma: [0,\infty) \rightarrow 
\mathbb{H}^n$ by 
\[
\gamma (t) := \exp(tY) =\exp\bigg(\sum_{j=1}^{2n+1}t\yy_j X_j\bigg) = (t\yy_1,\dots,t\yy_{2n+1}) =t\y.
\]
By the definition of the exponential map ,  $\gamma$
is the integral curve of the left-invariant vector field $Y$ such that 
$\gamma(0)=\mathbf{0}$ and $\gamma(1) =\y$. 
Then by the Fundamental Theorem of Calculus, 
we have
\begin{align*}
 f(\mathbf{x}\mathbf{y}) - f(\mathbf{x}) & =f(\x\gamma(1)) -f(\x \gamma(0)) 
  = \int_0^1 \frac{d}{dt} f(\x \gamma (t)) dt \\
& = \int_0^1 (Yf)(\x \gamma (t)) dt 
 = \int_0^1 \sum_{j=1}^{2n+1}\yy_j(X_jf)(\x \gamma (t)) dt\\
 &= \int_0^1 \sum_{j=1}^{2n+1}\yy_j(X_jf)\big(\x (t\y)\big) dt.
\end{align*}
The proof of the lemma is thus complete.
\end{proof}

\section{Heat semigroup characterization of $\Lambda^\alpha_L(\mathbb{H}^n)$} \label{sec:proof of main}
The main purpose of this section is to prove Theorem \ref{thm:main-1}. Our proof adopts the idea from \cite{DT}.

\subsection{Some auxiliary results} In this subsection we present some auxiliary results that will
be useful in the proof of Theorem \ref{thm:main-1}.

A function $\omega$ is said to a rapidly decaying function on $\mathbb{H}^n$, if for any $N >0$, 
there exists a constant $C_N$ such that for all $\x \in \mathbb{H}^n$, 
\begin{align*}
  |\omega (\x)| \leq C_N (1 +|\x|)^{-N}  
\end{align*}
For any $t >0$, let $\omega_t (\x):= t^{-Q/2}\omega\big(\delta_{t^{-1/2}} (\x) \big)$. 

\begin{lemma} \label{lem:AOI} {\rm (see \cite[Proposition 1.20]{FS})} 
Let $\omega$ be a rapidly decaying function on $\mathbb{H}^n$ and let $f \in L^1(\mathbb{H}^n)$. Then for almost
every $\x  \in \mathbb{H}^n$,
\[
\lim_{t \rightarrow 0} f \ast \omega_t (\x) =f(\x).
\]
\end{lemma}

\begin{lemma}\label{lem:limitation}
Assume that $f$ is a measurable function on $\mathbb{H}^n$ and there exists $t_0 > 0$ such that 
$\int_{\mathbb{H}^n} e^{-\frac{|\x|^2}{t_0}} f(\x) d\mu(\x) < \infty$. 
Then for almost every $\x \in \mathbb{H}^n$, 
\begin{align*}
\lim_{t \to 0} T_t^L f(\x) = f(\x).
\end{align*}
\end{lemma}
 
\begin{proof}
By \cite[Lemma 9]{LinLiu}, there exists a rapid decay function $\omega$ such that 
\begin{align*}
|K_t^L (\x, \y) - H_t(\y^{-1}\x)| \leq C \left( \frac{\sqrt{t}}{\rho(\x)}\right)^{2-\frac{Q}{q}} \omega_t (\y^{-1}\x), \quad
\text{for } \sqrt{t} \leq \rho(\x).
\end{align*}
Using this estimate and Lemma~\ref{lem:AOI}, and arguing similarly as in the proof of \cite[Lemma 2.3]{DT},
one can prove the assertion. 
\end{proof}

\begin{lemma} \label{lem:5.2}
Let $\alpha >0$.
If $\rho(\cdot)^{-\alpha}f \in L^\infty (\mathbb{H}^n)$, 
then   $f$ satisfies the heat size condition for $L$.  
\end{lemma}
 
\begin{proof}
From Remark \ref{rmk:growth of rho} we see that  $\rho(\cdot)^{-\alpha}f \in L^\infty (\mathbb{H}^n)$
implies $(1 + |\cdot|)^{-\alpha} f \in L^\infty (\mathbb{H}^n)$.
Hence, for every  $t >0$,
\begin{align*}
\int_{\mathbb{H}^n} e^{-\frac{|\x|^2}{t}} |f(\x)|d\mu(\x) &\leq C \int_{\mathbb{H}^n}\Big(1 + \frac{|\x|}{\sqrt{t}}\Big)^{-(Q+1 +\alpha) }|f(\x)|d\mu(\x) \\
&\leq C (1 +\sqrt{t})^{Q+1 +\alpha}  \int_{\mathbb{H}^n} (1 + |x|)^{-(Q+1 +\alpha)} |f(\x)| d\mu(\x)  \\
& \leq C (1 +\sqrt{t})^{Q+1 +\alpha}  \int_{\mathbb{H}^n} (1 + |x|)^{-(Q+1)} d\mu(\x)  \\
& <\infty,
\end{align*}
where we used the elementary inequalities
\[
e^{-s} \leq C_N (1 + s)^{-N}, \quad s>0
\]
and 
\[
\Big(1 + \frac{a}{b}\Big)^{-1}
\leq (1 +b) (1 +a)^{-1},\quad a,b >0.
\]

To show that $\lim_{t \rightarrow \infty} \partial_t^\ell T_t^L f(\x) =0$ for every $\ell \in \mathbb{N}_0$ and every $\x \in \mathbb{H}^n$, we write
\begin{align} \label{eq:tLell}
\partial_t^\ell T_t^L f (\x) = \int_{\mathbb{H}^n} \partial_t^\ell K^L_t (\x,\y)f(\y)d\mu(\y) .
\end{align}
By  Lemma \ref{lem:kernel estimatiaon},   there exist constants $C, c >0$ such that
\begin{align*}
|\partial_t^\ell K_t^L (x,y)|& \leq C t^{-\ell -\frac{Q}{2}} e^{-\frac{|\y^{-1}\x|^2}{ct}} \left(1+  \frac{\sqrt{t}}{\rho(\x)} +  \frac{\sqrt{t}}{\rho(\y)}\right)^{-(3+\alpha)} \\
& \leq C t^{-\ell -\frac{Q }{2}}\left( 1 +\frac{|\y^{-1}\x|}{\sqrt{t}}\right)^{-(Q +1+\alpha)}\left( \frac{\sqrt{t}}{\rho(\x)}\right)^{-(3+\alpha)}. 
\end{align*}
Using the triangle inequality \eqref{eq:triangle}, we have
\begin{align*}
1 + \frac{|\y|}{\sqrt{t}} \leq 1 +\frac{|\y^{-1}\x|}{\sqrt{t}} +\frac{\x}{\sqrt{t}} \leq \left(1 +\frac{|\y^{-1}\x|}{\sqrt{t}}\right)  \left(1 +\frac{|\x|}{\sqrt{t}}\right), 
\end{align*}
and hence 
\begin{align*}
\left( 1 +\frac{|\y^{-1}\x|}{\sqrt{t}}\right)^{-(Q +1+\alpha)} &\leq \left( 1 +\frac{|\y|}{\sqrt{t}}\right)^{-(Q +1+\alpha)}\left( 1 +\frac{|\x|}{\sqrt{t}}\right)^{Q +1+\alpha}\\
 & \leq (1 +\sqrt{t})^{Q + 1+\alpha} ( 1 + |\y| )^{-(Q +1 +\alpha)}\left( 1 +\frac{|\x|}{\sqrt{t}}\right)^{Q +1 +\alpha}.
\end{align*}
It follows that
\begin{align*}
|\partial_t^\ell K_t^L (\x,\y)|\leq   C t^{-\ell -\frac{Q }{2}}(1 +\sqrt{t})^{Q +1 +\alpha} 
( 1 + |\y| )^{-(Q +1+\alpha)}\left( 1 +\frac{|\x|}{\sqrt{t}}\right)^{Q +1+\alpha}\left( \frac{\sqrt{t}}{\rho(\x)}\right)^{-(3+\alpha)}.
\end{align*}
Inserting this into \eqref{eq:tLell}, and using that $(1+|\cdot|)^{-\alpha}f \in L^\infty(\mathbb{H}^n)$, we obtain
\begin{align} \label{eq:rewr}
   | \partial_t^\ell T_t^L f (\x) | &\leq C  t^{-\ell - \frac{Q }{2}} (1 +\sqrt{t})^{Q +1 +\alpha} \left( \frac{\sqrt{t}}{\rho(\x)}\right)^{-(3+\alpha)}  \left( 1 + \frac{|\x|}{\sqrt{t}}\right)^{Q +1 +\alpha}\int_{\mathbb{H}^n} 
      \frac{|f(\y)|}{(1 +|y|)^{Q+1 +\alpha} }d\mu(\y) \nonumber \\
     & \leq C   t^{-\ell - \frac{Q}{2}} (1 +\sqrt{t})^{Q +1 +\alpha} \left( \frac{\sqrt{t}}{\rho(\x)}\right)^{-(3+\alpha)}  \left( 1 + \frac{|\x|}{\sqrt{t}}\right)^{Q +1+\alpha}.
\end{align}

If $t > \max \{1, |\x|^2\}$, then $1 +\sqrt{t} \leq 2\sqrt{t}$ and  $ 1 + \frac{|\x|}{\sqrt{t}} \leq 2$. Hence from 
\eqref{eq:rewr} we see that for
$t >\max \{1, |\x|^2\}$
\begin{align*}
 | \partial_t^\ell T_t^L f (\x) | \leq C t^{-\ell} t^{\frac{1+\alpha}{2}} t^{-\frac{3+\alpha}{2}} \rho(\x)^{-(2+\alpha)}
 =t^{-\ell -1}\rho(\x)^{-(3+\alpha)}.
\end{align*}
Since  $\rho(x) < \infty$ (cf.  Lemma \ref{lem:rho finite}), it follows that 
$\lim_{t \rightarrow \infty} \partial_t^\ell T_t^L f (\x) =0$.
\end{proof}

\begin{lemma}\label{lem:mixed weighted estimates}
Let $\alpha > 0$, $k :=\lfloor \frac{\alpha}{2}\rfloor +1$ and
$f \in \Gamma_L^{\alpha/2} (\mathbb{H}^n)$. 
Then for any $\ell, m \in \mathbb{N}_0$  with $\frac{m}{2} + \ell \ge k$,  there exists a constant $C_{m,\ell} > 0$
such that
\begin{align*}
\big\|\rho(\cdot)^{-m} \partial_t^\ell T_t^L f \big\|_{L^\infty(\mathbb{H}^n)} \le C_{m,\ell}  t^{-(\frac{m}{2} + \ell) + \frac{\alpha}{2}} \|f\|_{\Gamma_L^{\alpha/2} (\mathbb{H}^n)}.
\end{align*}
\end{lemma}
\begin{proof}
First we consider the case $ \ell\ge k$, by the semigroup property, we have
\begin{equation} \label{eq:semigroup property}
\begin{split}
\partial_t^\ell T_t^L f & = (-L)^\ell  T_t f = (-L)^\ell  T_{t/2} T_{t/2} f =  (-L)^{\ell-k} T_{t/2} \big( (-L)^{k} T_{t/2}f\big)\\
& =2^k (-L)^{\ell-k} T_{t/2} \big(\partial_t^k T_{t/2}f\big) =2^\ell \partial_t^{\ell -k}T_{t/2} \big(\partial_t^k T_{t/2}f\big).
\end{split}
\end{equation}
Hence
\begin{align} \label{eq:aa}
    \left| \rho(\x)^{-m}\partial_t^\ell T_t^L f(\x) \right| \leq 2^\ell \rho(\x)^{-m} \int_{\mathbb{H}^n} \left| \partial_t^{\ell -k}  K_{t/2}^L(\x,\y) \right| 
    \left|\partial_t^k T_{t/2}^L f(\y) \right| d\mu(\y) .
    \end{align}
Note that 
\begin{align} \label{eq:bb}
\left|\partial_t^k T_{t/2}^L f(\y) \right|    = 2^{-k}\left|\partial_u^k T_{u}^L f(\y)\big|_{u =t/2} \right|
\leq 2^{-k} \|f\|_{\Gamma^{\alpha/2}_L(\mathbb{H}^n)} (t/2)^{-k +\frac{\alpha}{2}}.
\end{align}
On the other hand, by Lemma \ref{lem:kernel estimatiaon}, 
\begin{align} \label{eq:cc}
  \left| \partial_t^{\ell -k}  K_{t/2}^L(\x,\y) \right|   & = 2^{-(\ell -k)} \left| \partial_u^{\ell -k}  K_{u}^L(\x,\y)\big|_{u =t/2} \right|  \nonumber \\
  & \leq C_m 2^{-(\ell -k)}(t/2)^{-\frac{Q}{2} - \ell + k}\exp \left(-\frac{|\y^{-1}\x|^2}{ct/2} \right) \left(1 + \frac{\sqrt{t/2}}{\rho(\x)}  + \frac{\sqrt{t/2}}{\rho(\y)}\right)^{-m} \nonumber \\
  & \leq C_{m,\ell} 2^{-(\ell -k)} t^{-\frac{Q}{2} - \ell + k}\exp \left(-\frac{|\y^{-1}\x|^2}{ct/2} \right) \left(\frac{\rho(\x)}{\sqrt{t}} \right)^{m}.
\end{align}
Inserting \eqref{eq:bb} and \eqref{eq:cc} into \eqref{eq:aa}, it follows that
\begin{align*}
    \left| \rho(\x)^{-m}\partial_t^\ell T_t^L f(\x) \right| &\leq C_{m,\ell} t^{-k +\frac{\alpha}{2}}t^{ - \ell + k}t^{-\frac{m}{2}}\|f\|_{\Gamma^{\alpha/2}_L(\mathbb{H}^n)} \int_{\mathbb{H}^n}t^{-\frac{Q}{2}}\exp \left(-\frac{|\y^{-1}\x|^2}{ct/2} \right) d\mu(\y)\\
    & \leq C_{m,\ell}  t^{ -(\ell +\frac{m}{2})+\frac{\alpha}{2}}  \|f\|_{\Gamma^{\alpha/2}_L(\mathbb{H}^n)} ,
    \end{align*}
as desired.

Next we consider the case $\ell < k$. By the result proved in the previous case, we have
\begin{align} \label{eq:for k}
  \left| \rho(\x)^{-m}\partial_t^k T_t^L f(\x) \right| \leq C  t^{ -(k +\frac{m}{2})+\frac{\alpha}{2}}  \|f\|_{\Gamma^{\alpha/2}_L(\mathbb{H}^n)} .
\end{align}
Since $f \in \Gamma_L^{\alpha/2} (\mathbb{H}^n)$,  
$f$ satisfies the heat size condition for $L$. Hence for every $j \in \mathbb{N}_0$,
\[
\lim_{u \rightarrow \infty} \partial_u^j T_u^L f(\x) =0.
\]
Now using the Newton-Leibniz formula and \eqref{eq:for k}, we have
\begin{align*}
  \left| \rho(\x)^{-m}\partial_t^{k -1} T_t^L f(\x) \right|& = \left| \lim_{u \rightarrow \infty} \partial_u^j T_u^L f(\x)
  - \int_t^{\infty} \rho(\x)^{-m}\partial_s^{k} T_s^L f(\x)  ds \right|  \\
  & \leq C \int_t^{\infty}   s^{ -(k +\frac{m}{2})+\frac{\alpha}{2}}    ds \|f\|_{\Gamma^{\alpha/2}_L(\mathbb{H}^n)} \\
  &\leq C   t^{ -(k -1 +\frac{m}{2})+\frac{\alpha}{2}}  \|f\|_{\Gamma^{\alpha/2}_L(\mathbb{H}^n)}.
\end{align*}
 Repeating this process $k -\ell$ times, we get
 \begin{align*}
  \left| \rho(\x)^{-m}\partial_t^{\ell} T_t^L f(\x) \right|
  \leq C   t^{ -(\ell +\frac{m}{2})+\frac{\alpha}{2}}  \|f\|_{\Gamma^{\alpha/2}_L(\mathbb{H}^n)},
\end{align*}
as desired.
\end{proof}

Using Lemma \ref{lem:mixed weighted estimates} and arguing similarly as the proof of \cite[Proposition 2.9]{DT},
we may derive the following result.
\begin{proposition} \label{growth control of heat Lipschitz}
Let $\alpha >0$. If $f \in \Gamma_L^\alpha (\mathbb{H}^n)$, then 
$\rho(\cdot)^{-\alpha }f \in L^\infty (\mathbb{H}^n)$. Moreover, there exists
 a constant $C$ such that for all $f \in \mathcal{F}_L(\mathbb{H}^n)$,
\begin{align*}
 \|\rho(\cdot)^{-\alpha }f \|_{L^\infty (\mathbb{H}^n)}  \leq C \|f\|_{\Gamma_L^{\alpha/2} (\mathbb{H}^n)}.
\end{align*}
\end{proposition}

\begin{lemma} \label{lem:sdsds}
Let $f$ be a measurable function on $\mathbb{H}^n$ such that, for every $t >0$,
$\int_{\mathbb{H}^n} e^{-\frac{|\x|^2}{t}}f(\x)d\x <\infty$. Then, for almost every $\x \in \mathbb{H}^n$, we have
\begin{align} \label{eq:aeconverge}
\lim_{t \rightarrow 0}T_tf(\x) =f(\x).
\end{align}
Moreover, the function $(t,\x) \mapsto T_tf (\x)$ belongs to $C^\infty ((0,\infty) \times \mathbb{H}^n)$.
\end{lemma}
\begin{proof}
Assertion \eqref{eq:aeconverge} is special case of Lemma \ref{lem:limitation} (with $V\equiv 0$).  
Let us prove the second assertion.  Fix an arbitrary $t_0 >0$ and an arbitrary $\x_0 \in \mathbb{H}^n$,
and let $U$ be a (small) neighborhood of $(t_0,\x_0)$ in $(0,\infty) \times \mathbb{H}^n$. It suffices to
to show that the function 
\[
(t,\x) \mapsto T_t f(\x) = \int_{\mathbb{H}^n} H_t (\y^{-1}\x) f(\y)d\mu(\y)
\]
is smooth in $U$. By the Lebesgue dominated convergence theorem, it suffices to show that for any $\ell \in \mathbb{N}_0$,
any $m \in \mathbb{N}_0$, and any $(i_1,\dots, i_m) \in \{1,\dots, 2n+1\}^m$, there exists a constant $C = C_{f,\ell, m, U}$,
such that for all $(t, \x) \in U$,
\begin{align*}
\int_{\mathbb{H}^n } \left| (\partial_t^\ell X_{i_1} \cdots X_{i_m} H_t)(\y^{-1}\x) f(\y) \right| d\mu(\y) 
\leq C.
\end{align*}

By Lemma \ref{lem:heat kernel of sub-Laplacian}, we have
\begin{align*}
\big|(\partial_t^\ell X_{i_1}\cdots X_{i_m} H_t) (\y^{-1}\x)\big| \leq C_{\ell,m} t^{-\ell -\frac{Q + \sigma (i_1) +\cdots +\sigma (i_m)}{2}}e^{-\frac{|\y^{-1}\x|^2}{ct}}.
\end{align*}
It follows that
\begin{equation} \label{eq:wwww}
\begin{split}
&\int_{\mathbb{H}^n } \left| (\partial_t^\ell X_{i_1} \cdots X_{i_m} H_t)(\y^{-1}\x) f(\y) \right| d\mu(\y) \\
 &\quad\quad\quad\quad\quad  \leq C t^{- \ell -\frac{Q + \sigma(i_1) +\cdots + \sigma(i_m)}{2}}
 \int_{\mathbb{H}^n}e^{-\frac{|\y^{-1}\x|^2}{ct}}|f(\y)|d\mu(\y).
 \end{split}
\end{equation}

If $|\y| \geq 2|\x|$, then by the triangle inequality \eqref{eq:triangle}, we have
$|\y^{-1}\x| \geq |\y| - |\x| \geq \frac{1}{2}|\y|$, and thus
$e^{-\frac{|\y^{-1}\x|^2}{ct}} \leq  e^{-\frac{|\y|^2}{2ct}}$. If, instead, $|\y| \leq 2|\x|$, then
$e^{-\frac{|\y^{-1}\x|^2}{ct}} \leq 1 \leq e^{4|\x|^2}e^{-|\y|^2}$. Using these estimates and \eqref{eq:wwww},
we have, for every $(t,\x) \in U$,
\begin{align*}
&\int_{\mathbb{H}^n } \left| (\partial_t^\ell X_{i_1} \cdots X_{i_m} H_t)(\y^{-1}\x) f(\y) \right| d\mu(\y) \\
 &\leq C_{\ell,m} t^{- \ell -\frac{Q + \sigma(i_1) +\cdots + \sigma(i_m)}{2}}
 \left(\int_{|\y| \geq 2|\x|}e^{-\frac{|\y|^2}{2ct}}|f(\y)|d\mu(\y) +e^{4|\x|^2}\int_{|\y|<2|\x|}e^{-|\y|^2}|f(\y)|d\mu(\y)\right)\\
 &\leq C_{f,\ell,m,U},
\end{align*}
as desired.
\end{proof}

\subsection{Proof of Theorem \ref{thm:main-1}}
We introduce an auxiliary class of Lipschitz functions as follows.
\begin{definition}
{\rm Let $\alpha >0$ and $k:= \lfloor \alpha/2 \rfloor +1$. The space $\widetilde{\Gamma}^{\alpha/2}(\mathbb{H}^n)$
is defined as as 
\begin{align*}
\widetilde{\Gamma}^{\alpha/2}(\mathbb{H}^n) := \Big\{f : (1 + |\cdot|)^{-\alpha} f\in L^\infty (\mathbb{H}^n) 
 \text{ and } \sup_{t >0} t^{k-\frac{\alpha}{2}}\big\|\partial_t^k T_t f \big\|_{L^\infty (\mathbb{H}^n)} <\infty\Big\}.
\end{align*}
We} endow this class with the norm
\begin{align*}
 \|f\|_{\widetilde{\Gamma}^{\alpha/2} (\mathbb{H}^n)} :=  \|  (1 + |\cdot|)^{-\alpha} f \|_{L^\infty(\mathbb{H}^n)} +   \sup_{t >0} t^{k-\frac{\alpha}{2}}\big\|\partial_t^k T_t f \big\|_{L^\infty (\mathbb{H}^n)} .
\end{align*}
\end{definition}

\begin{remark}
{\rm  If $(1 + |\cdot|)^{-\alpha }f \in L^\infty (\mathbb{H}^n)$, then $\partial_t^{\ell}X_{i_1}\cdots X_{i_m}T_t f$ is well defined for every $\ell \in \mathbb{N}_0$, every $m \in\mathbb{N}_0$, and every $(i_1,\dots, i_m) \in \mathbb{N}_0^m$. Indeed, from the proof of Lemma~\ref{lem:5.2} we see that  $(1 + |\cdot|)^{-\alpha }f \in L^\infty (\mathbb{H}^n)$ implies   $\int_{\mathbb{H}^n} e^{-\frac{|\x|^2}{t}}f(\x)d\x <\infty$ for every $t >0$. Hence
it follows from Lemma \ref{lem:sdsds} that the function $(t,\x)\mapsto \partial_t^{\ell}X_{i_1}\cdots X_{i_m}T_t f(x)$ is smooth on $(0,\infty) \times \mathbb{H}^n$.}
\end{remark}

\begin{lemma} \label{lem:XjXl-0}
Let $0<\alpha <2$ and let $f$ be a measurable function on $\mathbb{H}^n$ 
such that $(1 + |\cdot|)^{-\alpha} f \in L^\infty (\mathbb{H}^n)$.
Then, for every $j, \ell \in \{1,\dots, 2n+1\}$ and every $\x \in \mathbb{H}^n$, 
$X_jX_\ell \partial_t T_t f$ is well defined, and 
\begin{align*}
\lim_{t \rightarrow \infty} \partial_tX_\ell X_j  T_t f(\x) =0.
\end{align*}
\end{lemma}
\begin{proof}
Note that 
\begin{align*}
\partial_t X_\ell X_j  T_t f(\x) = \int_{\mathbb{H}^n} ( X_\ell X_j \partial_t H_t) (y^{-1}x) f(\y) d\mu(\y)
\end{align*}
By Lemma \ref{lem:heat kernel of sub-Laplacian},
\begin{align*}
\big|  (\partial_t X_\ell X_j   H_t) (y^{-1}x)\big| \leq C t^{-1 -\frac{Q + \sigma(j) + \sigma(\ell)}{2}}  
e^{-\frac{|\y^{-1}\x|^2}{ct}}.
\end{align*}
Hence
\begin{align*}
\big|\partial_t X_\ell X_j  T_t f(\x)\big| &\leq C \|(1 +|\cdot|)^{-\alpha}f \|_{L^\infty(\mathbb{H}^n)}
t^{-1 -\frac{\sigma(j) +\sigma(\ell)}{2}} \int_{\mathbb{H}^n} t^{-\frac{Q}{2}}  e^{-\frac{|\y^{-1}\x|^2}{ct}} (1 + |\y|)^\alpha d\mu(\y).
\end{align*}
Using the triangle inequality \eqref{eq:triangle}, we have
\begin{align*}
&1 +|\y|  =  1 + |\y^{-1}| = 1 +  |y^{-1}\x \x^{-1}|  \leq 1 + |\y^{-1}x| + |\x^{-1}|
= 1 +  |\y^{-1}x| + |\x|\\
&\quad\quad\quad\quad  \leq (1 + |\y^{-1}\x|) (1 + \x) \leq 
\Big( 1 +   \frac{|\y^{-1}\x|}{\sqrt{t}} \Big) (1 + \sqrt{t})(1 +\x),
\end{align*}
it follows that
\begin{align*}
\big| \partial_tX_\ell X_j   T_t f(\x)\big| &\leq  C  \|(1 +|\cdot|)^{-\alpha}f \|_{L^\infty(\mathbb{H}^n)}
(1 + |\x|)^\alpha  t^{-1 -\frac{\sigma(j) +\sigma(\ell)}{2}}(1 +\sqrt{t})^\alpha \\
&\quad\quad\quad\quad \times \int_{\mathbb{H}^n} t^{-\frac{Q}{2}}  e^{-\frac{|\y^{-1}\x|^2}{ct}} \Big(1 +\frac{|\y^{-1}\x|}{\sqrt{t}}\Big)^\alpha d\mu(\y)\\
&\leq C  \|(1 +|\cdot|)^{-\alpha}f \|_{L^\infty(\mathbb{H}^n)}
(1 + |\x|)^\alpha  t^{-1 -\frac{\sigma(j) +\sigma(\ell)}{2}}(1 +\sqrt{t})^\alpha.
\end{align*}
Since $-1 -\frac{\sigma(j) +\sigma(\ell)}{2} + \frac{\alpha}{2} <0$, the above estimate implies 
$\lim_{t \rightarrow \infty}\partial_t X_\ell X_j  T_t f(\x) =0$.
\end{proof}

\begin{lemma} \label{lem:XjXl}
Let $0< \alpha <2$ let $f$ be a measurable function on $\mathbb{H}^n$ 
such that $(1 + |\cdot|)^{-\alpha} f \in L^\infty (\mathbb{H}^n)$.
There exists a constant $C$ such that for all $j, \ell \in \{1,\dots, 2n+1\}$,
\begin{align*}
\left\|X_\ell X_j T_t f\right\|_{L^\infty} \leq C  t^{\frac{\alpha-\sigma(j) -\sigma(\ell)}{2}} \sup_{s>0} s^{1-\frac{\alpha}{2}} \big\|\partial_s T_s f \big\|_{L^\infty (\mathbb{H}^n)},
\end{align*}
where $\sigma(i)$, $i =1,\dots, 2n+1$, denotes the homogeneous degree of $X_i$, cf. \eqref{eq:homogeneous degree of X}. 
\end{lemma}
\begin{proof}
We denote $A:= \sup_{s>0} s^{1-\frac{\alpha}{2}} \big\|\partial_s T_s f \big\|_{L^\infty (\mathbb{H}^n)}$.
Similarly to \eqref{eq:semigroup property}, we have
\[
\partial_t   X_\ell X_j T_t f = 2 X_\ell X_j H_{t/2} (\partial_t H_{t/2}f).
\]
Hence, for every $\x \in \mathbb{H}^n$,
\begin{align*}
\left|   \partial_t   X_\ell X_j T_t f(\x) \right| &=
 2\left| \int_{\mathbb{H}^n}  (X_\ell X_j  {H}_{t/2} )(\y^{-1} \x) (\partial_t  {H}_{t/2}f ) (\y)  d\mu(\y)\right|\\
& \leq 2 A
 \int_{\mathbb{H}^n} t^{\frac{\alpha}{2} -1}\left| (X_\ell X_j  {H}_{t/2} )(\y^{-1} \x) \right| d\mu(\y) \\
 &\leq 2A t^{\frac{\alpha}{2} -1}t^{-\frac{ \sigma (j) +\sigma(\ell)}{2}}\int_{\mathbb{H}^n} t^{-\frac{Q}{2}}{}e^{-\frac{|\y^{-1}x|}{ct}}   d\mu(\y)\\
  &\leq C A t^{\frac{ \alpha - \sigma (j) -\sigma(\ell)}{2} -1} .
\end{align*}
Since $\lim_{u \rightarrow \infty}   \partial_u   X_\ell X_j T_t f (\x) =0$ (cf. Lemma \ref{lem:XjXl-0}), we have 
\begin{align*}
 \left|     X_\ell X_j T_t f(\x) \right|& =\left| \lim_{u \rightarrow \infty} \partial_u   X_\ell X_j T_u f)(\x) -\int_t^\infty \partial_u   X_\ell X_j T_u f (\x) du \right|    \\
 &\leq \int_t^\infty \left|\partial_u   X_\ell X_j T_u f (\x) \right| du \\
 &\leq CA \int_t^\infty u^{\frac{ \alpha - \sigma (j) -\sigma(\ell)}{2} -1} du \\
 &\leq CA t^{\frac{ \alpha - \sigma (j) -\sigma(\ell)}{2} },
\end{align*}
where we used that $0< \alpha <2$ for the last inequality.  
\end{proof}

\begin{lemma} \label{characterizatoin of Lip heat}
Let $0< \alpha <2$. Then 
$f \in \widetilde{\Gamma}^{\alpha/2} (\mathbb{H}^n)$ if and only if
{\small \begin{align} \label{two finite}
   \sup_{|\y| >0}
      \frac{\|f(\x \y) + f(\x \y^{-1}) -2 f(\x)\|_{L^\infty(\mathbb{H}^n_{\x})}}{|\y|^\alpha} <\infty \ \text{ and }  \
 \|(1 + |\cdot|)^{-\alpha}f \|_{L^\infty(\mathbb{H}^n)} <\infty.
  \end{align}
Moreover}, 
\begin{align}  \label{eq:equivalent classical}
 \|f\|_{\widetilde{\Gamma}^{\alpha/2} (\mathbb{H}^n)} \asymp  \sup_{|\y|>0}
  \frac{\|f(\x \y) + f(\x\y^{-1}) -2 f(\x)\|_{L^\infty(\mathbb{H}^n_\x)}}{|\y|^\alpha} +  \|(1+|\cdot|)^{-\alpha}f \|_{L^\infty(\mathbb{H}^n)} .
\end{align}
\end{lemma}
\begin{proof}
We first show the ``$\gtrsim$'' direction of \eqref{eq:equivalent classical}.
Let $f \in \widetilde{\Gamma}^{\alpha/2} (\mathbb{H}^n)$. 
We write, for every $t >0$ and $\x,\y \in \mathbb{H}^n$,
\begin{equation} \label{eq:00000}
\begin{split}
&|f(\x \y) + f(\x\y^{-1}) -2 f(\x) | \\
&\quad\quad \leq \big|T_t f (\x\y) -f (\x\y)\big| +  
\big|T_t f (\x\y^{-1}) -f (\x\y^{-1})\big| +2 \big|T_t f(\x) -f(\x)\big| \\
&\quad\quad\quad  +  \Big|\big[(T_tf)(\x\y) -(T_tf)(\x)\big] 
+ \big[(T_tf)(\x \y^{-1}) -(T_tf)(\x)\big] \Big|
\end{split}
\end{equation}
By Lemma \ref{lem:limitation} we have, for every $y \in \mathbb{H}^n$ and for almost every $\x \in \mathbb{H}^n$,
\begin{align*}
    \lim_{t \rightarrow 0}T_t f(\x \y) =f(\x\y).
\end{align*}
It follows that
\begin{align*}
\big|T_t f(\x\y) -f(\x\y)\big|& = \left|\lim_{u \rightarrow 0} T_t f(\x\y) + \int_0^t  \partial_u T_u f(\x\y)du -f(\x\y) \right|   \\
& \leq  \int_0^t  \big| \partial_u T_u f(\x\y) \big| du \\
&\leq \|f\|_{\widetilde{\Gamma}^{\alpha/2}(\mathbb{H}^n)} \int_0^t  u^{-1 +\frac{\alpha}{2}} du \\
&=C\|f\|_{\widetilde{\Gamma}^{\alpha/2}(\mathbb{H}^n)}  t^{\frac{\alpha}{2}}.
\end{align*}
Hence, for every $\y \in \mathbb{H}^n$,
\begin{align} \label{eq:11111}
\|T_t f(\x\y) -f(\x\y)\|_{L^\infty (\mathbb{H}_{\x}^{n})} \leq C\|f\|_{\widetilde{\Gamma}^{\alpha/2}(\mathbb{H}^n)}  t^{\frac{\alpha}{2}}.
\end{align}
Similarly, we also have
\begin{align} \label{eq:22222}
  \|T_t f (\x\y^{-1}) -f (\x\y^{-1})\|_{L^\infty (\mathbb{H}_{\x}^{n})} +  \|T_t f (\x) -f (\x)\|_{L^\infty (\mathbb{H}_{\x}^{n})} \leq C\|f\|_{\widetilde{\Gamma}^{\alpha/2}(\mathbb{H}^n)} t^{\frac{\alpha}{2}}.
\end{align}


Recall that for any $\y =(\yy_1, \dots, \yy_{2n+1}) \in \mathbb{H}^{n}$, we have
$\y^{-1} = -\y = (-\yy_1,\dots, -\yy_{2n+1})$. 
Using this fact and Lemma~\ref{lem:FTC},
we have, for every $t >0$ and  every $\x,\y \in \mathbb{H}^n$,
\begin{align} \label{eq:difference}
&   \big[(T_tf)(\x\y) -(T_tf)(\x)\big] 
+ \big[(T_tf)(\x \y^{-1}) -(T_tf)(\x)\big]   \nonumber \\[1.5mm]
&\quad=   \int_0^1 \sum_{j =1}^{2n +1} \yy_j(X_j T_tf)\big(\x (s\y)\big) ds +
\int_0^1 \sum_{j =1}^{2n+1} (-\yy_j)(X_j T_t f)\big(\x(-s\y)\big) ds   \nonumber \\
&\quad =  \sum_{j=1}^{2n +1} \yy_j \int_0^1 \big[ ( X_j T_t f)\big(\x (s\y)\big) - 
  (X_j T_tf)\big(\x (-s\y)\big) \big] ds  \nonumber  \\
&\quad = \sum_{j=1}^{2n +1} \yy_j \int_0^1 \big[ ( X_j T_t f)\big(\x (s\y)\big) - 
  (X_j T_tf)\big(\x (-s\y)\big) \big] dt .
\end{align}
Using Lemma \ref{lem:FTC} again, we can write the last integrand as
\begin{align*}
& ( X_j T_t f)\big(\x (s\y)\big) - 
  (X_j T_tf)\big(\x (-s\y)\big) \\[2.5mm]
  &\quad =   \big[  ( X_j T_t f)\big(\x (s\y)\big) -  ( X_j T_t f) (\x )\big]
 -\big[ (X_j T_tf)\big(\x (-s\y)\big) -  ( X_j T_t f) (\x )  \big]\\
 &\quad = \int_0^1 \sum_{\ell=1}^{2n+1} s\yy_{\ell}( X_\ell X_j T_t f)\big(\x(rs \y)\big)  dr - 
\int_0^1 \sum_{\ell =1}^{2n +1} (-s \yy_\ell) (X_\ell X_j T_t f)\big(\x (-rs\y)\big) dr\\
 &\quad = \int_0^1 \sum_{\ell=1}^{2n+1} s\yy_{\ell}( X_\ell X_j T_t f)\big(\x(rs \y)\big)  dr +
\int_{-1}^0 \sum_{\ell =1}^{2n +1} s \yy_\ell (X_\ell X_j T_t f)\big(\x (rs\y)\big) dr\\
 &\quad = \int_{-1}^1 \sum_{\ell=1}^{2n+1} s\yy_{\ell}( X_\ell X_j T_t f)\big(\x(rs \y)\big)  dr.
\end{align*}
Inserting this into \eqref{eq:difference} gives
\begin{align*}
&  \big[(T_tf)(\x\y) -(T_tf)(\x)\big] 
+ \big[(T_tf)(\x \y^{-1}) -(T_tf)(\x)\big]  \\[1.5mm] 
& \quad = \sum_{j=1}^{2n+1}\sum_{\ell =1}^{2n+1} \yy_j \yy_\ell \int_0^1 s\left( \int_{-1}^1 ( X_\ell X_j  T_tf )\big(\x (rs \y)\big) dr\right)ds .
\end{align*}
By Lemma \ref{lem:XjXl}, we have, for all $s \in [0,1]$, all $r \in [-1,1]$ and all $\x, \y \in \mathbb{H}^n$, 
\begin{align*}
 \left|  ( X_kX_j  T_tf )\big(\x (rs \y)\big) \right|  &\leq C  t^{\frac{\alpha-\sigma(j) -\sigma(\ell)}{2}} \sup_{u>0} u^{1-\frac{\alpha}{2}} \big\|\partial_u T_u f \big\|_{L^\infty (\mathbb{H}^n)} \\
 &\leq C t^{\frac{\alpha-\sigma(j) -\sigma(\ell)}{2}}  \|f\|_{\widetilde{\Gamma}_L^{\alpha/2}(\mathbb{H}^n)}.
\end{align*}
Hence
\begin{align*}
&  \Big| \big[(T_tf)(\x\y) -(T_tf)(\x)\big] 
+ \big[(T_tf)(\x \y^{-1}) -(T_tf)(\x)\big]  \Big| \\
&\quad\quad\quad \leq C \sum_{j=1}^{2n+1}\sum_{\ell =1}^{2n+1} |\yy_j||\yy_\ell|t^{\frac{\alpha-\sigma(j) -\sigma(\ell)}{2}}  \|f\|_{\widetilde{\Gamma}_L^{\alpha/2}(\mathbb{H}^n)}\\
&\quad\quad\quad \leq C \sum_{j=1}^{2n+1}\sum_{\ell =1}^{2n+1} |\yy_j||\yy_\ell|t^{\frac{\alpha-\sigma(j) -\sigma(\ell)}{2}}  \|f\|_{\widetilde{\Gamma}_L^{\alpha/2}(\mathbb{H}^n)}.
\end{align*}
 Since $|\yy_i| \lesssim |\y|^{\sigma(i)}$ for every $i \in \{1,\dots, 2n+1\}$, it follows that
  for every $\y \in \mathbb{H}^n$,
\begin{equation} \label{eq:33333}
\begin{split}
& \Big\| \big[(T_tf)(\x\y) -(T_tf)(\x)\big] 
+ \big[(T_tf)(\x \y^{-1}) -(T_tf)(\x)\big]  \Big\|_{L^\infty(\mathbb{H}^n_\x)}  \\
&\quad\quad\quad \leq C\sum_{j=1}^{2n+1}\sum_{\ell =1}^{2n+1}|\y|^{\sigma(j) + \sigma(\ell)}t^{\frac{\alpha-\sigma(j) -\sigma(\ell)}{2}} \|f\|_{\widetilde{\Gamma}^{\alpha /2}(\mathbb{H}^n)}    
\end{split}
\end{equation}
Now we let $t =|\y|^2$. Then combining \eqref{eq:00000}, \eqref{eq:11111}, \eqref{eq:22222} and 
\eqref{eq:33333}, we obtain 
 \begin{align*}
\|f(\x \y) + f(\x \y^{-1}) -2 f(\x)\|_{L^\infty(\mathbb{H}^n_{\x})}
 \leq C\|f\|_{\widetilde{\Gamma}^{\alpha/2} (\mathbb{H}^n)} |\y|^\alpha,
 \end{align*}
which implies the ``$\gtrsim$'' direction of 
 \eqref{eq:equivalent classical}.

 For showing the ``$\lesssim$'' direction of 
 \eqref{eq:equivalent classical}, we assume that \eqref{two finite} holds. Since ${H}_t (\x) = {H}_t (\x^{-1})$, 
 we have
 \begin{align*}
  \int_{ \mathbb{H}^n} (\partial_t {H}_t) (\y) f(\x\y)d\y  =   
  \int_{ \mathbb{H}^n} (\partial_t {H}_t) (\y^{-1}) f(\x\y^{-1})d\y = \int_{ \mathbb{H}^n} (\partial_t {H}_t) (\y) f(\x\y^{-1})d\y.
 \end{align*}
 Meanwhile, since $\int_{\mathbb{H}^n}  {H}_t(\y)d\y =1$, 
 \[
 \int_{\mathbb{H}^n} \partial_t  {H}_t (\y)d\y = \partial_t\int_{\mathbb{H}^n}    {H}_t (\y)d\y =0.
 \]
 It follows that, for almost every $\x \in \mathbb{H}^n$,
 \begin{align*}
 |\partial_t  {H}_t f(\x)| & =  \frac{1}{2} \left|\int_{\mathbb{H}^n}  {H}_t (\y) \big[ f(\x \y) + f(\x \y^{-1})-2f(\x)\big]d\mu(\y)\right|   \\
 & \leq C  \sup_{|\mathbf{w}| >0}
  \frac{\|f(\mathbf{z} \mathbf{w}) + f(\mathbf{z}\mathbf{w}^{-1}) -2 f(\mathbf{z})\|_{L^\infty(\mathbb{H}^n_{\mathbf{z}})}}{|\mathbf{w}|^\alpha} \int_{\mathbb{H}^n} |(\partial_t{H}_t) (\mathbf{y})||\mathbf{y}|^\alpha d\mu(\mathbf{y}).
 \end{align*}
By Lemma \ref{lem:heat kernel of sub-Laplacian}, 
\begin{align*}
\int_{\mathbb{H}^n} |(\partial_t {H}_t) (\mathbf{y})||\mathbf{y}|^\alpha d\mu(\mathbf{z}) & \leq C  \int_{\mathbb{H}^n}
t^{-1-\frac{Q}{2}}e^{-\frac{|\mathbf{y}|^2}{ct}}|\mathbf{y}|^\alpha d\mu(\mathbf{y})\\
& = Ct^{-1+ \frac{\alpha}{2}} \int_{\mathbb{H}^n}
t^{-\frac{Q}{2}}e^{-\frac{|\mathbf{y}|^2}{ct}}\left(\frac{|\mathbf{y}|^2}{ct}\right)^{\frac{\alpha}{2}} d\mu(\mathbf{y})\\
& \leq C t^{-1+ \frac{\alpha}{2}} \int_{\mathbb{H}^n}
t^{-\frac{Q}{2}}e^{-\frac{|\mathbf{y}|^2}{2ct}} d\mu(\mathbf{y})\\
&\leq C t^{-1+ \frac{\alpha}{2}}.
\end{align*}
Hence for almost every $\x \in \mathbb{H}^n$,
 \begin{align*}
 |\partial_t  {H}_t f(\x)|  \leq  C  t^{-1 +\frac{\alpha}{2}}  \sup_{|\mathbf{w}| >0}
  \frac{\|f(\mathbf{z} \mathbf{w}) + f(\mathbf{z}\mathbf{w}^{-1}) -2 f(\mathbf{z})\|_{L^\infty(\mathbb{H}^n_{\mathbf{z}})}}{|\mathbf{w}|^\alpha}.
 \end{align*}
Consequently,
\begin{align*}
\|f\|_{\widetilde{\Gamma}^{\alpha/2}(\mathbb{H}^n)} &= \sup_{t >0} t^{1 -\frac{\alpha}{2}} \|\partial_t T_tf\|_{L^\infty (\mathbb{H}^n)} +
 \|(1 +|\cdot|)^{-\alpha}f\|_{L^\infty (\mathbb{H}^n)} \\
&\leq C \left( \sup_{|\mathbf{w}| >0}
  \frac{\|f(\mathbf{z} \mathbf{w}) + f(\mathbf{z}\mathbf{w}^{-1}) -2 f(\mathbf{z})\|_{L^\infty(\mathbb{H}^n_{\mathbf{z}})}}{|\mathbf{w}|^\alpha} + \|(1 +|\cdot|)^{-\alpha}f\|_{L^\infty (\mathbb{H}^n)}\right) ,
\end{align*}
as desired.
\end{proof}

\begin{lemma} \label{lem:rapid V} {\rm (see \cite[Lemma 6]{LinLiu})}
Suppose $\omega$ is a nonnegative rapidly decaying function. Then there exists
a constant $C >0$ such that 
\begin{align*}
  \int_{\mathbb{H}^n} V(\y)\omega_t (\y^{-1}\x) d\mu(\y) \leq C \frac{1}{t} \left(\frac{\sqrt{t}}{\rho(\x)} \right)^{2-\frac{Q}{q}},
  \quad \text{whenever } \sqrt{t} \leq \rho(\x). 
\end{align*}
\end{lemma}

 Using Lemma \ref{lem:rapid V} and arguing similarly as \cite[Theorem 3.11]{DT}, one can derive 
 the following result. We omit the details here. 

\begin{proposition} \label{comparison of heat semigroups}
Let $0< \alpha \leq 2- Q /q$, and let $f$ be a function such that 
$\rho(\cdot)^{-\alpha} f \in L^\infty(\mathbb{H}^n)$. Then there exists a constant $C$ such that for all
$t>0$, \begin{align*}
\|\partial_t T_t^L f - \partial_t T_t f\|_{L^\infty (\mathbb{H}^n)}
\leq C t^{-1+\frac{\alpha}{2}}  \|\rho(\cdot)^{-\alpha} f \|_{L^\infty(\mathbb{H}^n)}.
\end{align*}
Consequently,  
\begin{align*}
\sup_{t >0 } t^{1-\frac{\alpha}{2}} \|\partial_t T_t^L f \|_{L^\infty (\mathbb{H}^n)}
\leq  \sup_{t >0 } t^{1-\frac{\alpha}{2}} \|\partial_t T_t f \|_{L^\infty (\mathbb{H}^n)} + C  \|\rho(\cdot)^{-\alpha} f \|_{L^\infty(\mathbb{H}^n)}
\end{align*}    
and
\begin{align*}
\sup_{t >0 } t^{1-\frac{\alpha}{2}} \|\partial_t T_t f \|_{L^\infty (\mathbb{H}^n)}
\leq  \sup_{t >0 } t^{1-\frac{\alpha}{2}} \|\partial_t T^L_t f \|_{L^\infty (\mathbb{H}^n)} + C  \|\rho(\cdot)^{-\alpha} f \|_{L^\infty(\mathbb{H}^n)}.
\end{align*}  
\end{proposition}
 
 \medskip
 
We are now in a position to complete the proof of  Theorem \ref{thm:main-1}. 
\begin{proof}[Completing the proof of Theorem \ref{thm:main-1}]
First we  show that $\Lambda^\alpha_L (\mathbb{H}^n) \subset \Gamma_L^{\alpha/2} (\mathbb{H}^n)$ and the inclusion map 
is continuous.
For any $f \in \Lambda_L^\alpha (\mathbb{H}^n)$, we have 
\begin{align} \label{eq:sec diff}
   \sup_{|\y| >0}
      \frac{\|f(\x \y) + f(\x \y^{-1}) -2 f(\x)\|_{L^\infty(\mathbb{H}^n_{\x})}}{|\y|^\alpha} <\infty 
  \end{align}
  and
  \begin{align} \label{eq:control rho}
  \|\rho(\cdot)^{-\alpha} f \|_{L^\infty(\mathbb{H}^n)} <\infty.
  \end{align}
  Since $\rho(\x) \leq C(1 +|\x|)$ (see Remark \ref{rmk:growth of rho}), the property \eqref{eq:control rho} 
  implies that
  \begin{align} \label{eq:control 1+x}
    \|(1 +|\cdot|)^{-\alpha}f \|_{L^\infty (\mathbb{H}^n)} <\infty.
  \end{align}
In view of Lemma \ref{characterizatoin of Lip heat},  conditions \eqref{eq:sec diff} and \eqref{eq:control 1+x}
imply that $f\in \widetilde{\Gamma}^{\alpha/2}(\mathbb{H}^n)$. Moreover,  
\begin{align*}
&\sup_{t >0}t^{1-\frac{\alpha}{2}}  \|\partial_t T_t f  \|_{L^\infty (\mathbb{H}^n)} \\
&\quad\quad \leq
\|f\|_{\widetilde{\Gamma}^{\alpha/2}(\mathbb{H}^n)} \\
&\quad\quad =  \sup_{|\y| >0}
      \frac{\|f(\x \y) + f(\x \y^{-1}) -2 f(\x)\|_{L^\infty(\mathbb{H}^n_{\x})}}{|\y|^\alpha} +  \|(1+|\cdot|)^{-\alpha}f \|_{L^\infty(\mathbb{H}^n)} \\
  &\quad\quad  \leq   \sup_{|\y| >0}
      \frac{\|f(\x \y) + f(\x \y^{-1}) -2 f(\x)\|_{L^\infty(\mathbb{H}^n_{\x})}}{|\y|^\alpha} + C \|\rho(\cdot)^{-\alpha}f 
  \|_{L^\infty(\mathbb{H}^n)} \\
  &\quad\quad  \leq C \|f\|_{\Lambda^\alpha_L(\mathbb{H}^n)}.
\end{align*}
Finally, applying Proposition \ref{comparison of heat semigroups} and Proposition \ref{growth control of heat Lipschitz},
we see that 
\begin{align*}
\|f\|_{\Gamma_L^{\alpha/2} (\mathbb{H}^n)} &= \sup_{t >0 } t^{1-\frac{\alpha}{2}} \|T_t^L f\|_{L^\infty (\mathbb{H}^n)} \\
&\leq  \sup_{t >0 } t^{1-\frac{\alpha}{2}} \|\partial_t T_t f\|_{L^\infty (\mathbb{H}^n)} + C  \|\rho(\cdot)^{-\alpha} f  \|_{L^\infty(\mathbb{H}^n)}\\
&\leq\|  f \|_{\Lambda_L^\alpha (\mathbb{H}^n)} +C\|  f \|_{\Lambda_L^\alpha (\mathbb{H}^n)}\\
& \leq C\|  f \|_{\Lambda_L^\alpha (\mathbb{H}^n)}.
\end{align*}
Hence   $\Lambda^\alpha_L (\mathbb{H}^n) \subset \Gamma_L^\alpha (\mathbb{H}^n)$ and the inclusion map 
is continuous.

\medskip
Next we show that $\Gamma_L^{\alpha/2} (\mathbb{H}^n) \subset \Lambda_L^\alpha (\mathbb{H}^n)$ and the inclusion
map is continuous. Let $f \in \Gamma_L^{\alpha/2} (\mathbb{H}^n)$. By Lemma \ref{growth control of heat Lipschitz}, 
we have $\rho(\cdot)^{-\alpha} f \in L^\infty (\mathbb{H}^n)$ and
\begin{align} \label{eq:rho alpha finite}
 \|\rho(\cdot)^{-\alpha}f\|_{L^\infty (\mathbb{H}^n)} \leq C \|f\|_{\Gamma_L^{\alpha/2} (\mathbb{H}^n)}.   
\end{align}
This estimate also implies
\begin{align} \label{eq:1+x finite}
 \|(1 +|\cdot|)^{-\alpha}f\|_{L^\infty (\mathbb{H}^n)} \leq C \|f\|_{\Gamma_L^{\alpha/2} (\mathbb{H}^n)}.   
\end{align}
since $\rho(\x)\leq C(1 +|\cdot|)$ (see Remark \ref{rmk:growth of rho}).
With the condition \eqref{eq:rho alpha finite}, we may use Proposition \ref{comparison of heat semigroups} to deduce that
\begin{align*}
 \sup_{t >0}t^{1-\frac{\alpha}{2}}\|T_t f\|_{L^\infty (\mathbb{H}^n)}
 &\leq  \sup_{t >0}t^{1-\frac{\alpha}{2}}\|T_t^L f\|_{L^\infty (\mathbb{H}^n)} + C \|\rho(\cdot)^{-\alpha}f\|_{L^\infty (\mathbb{H}^n)}\\
 &\leq \|f\|_{\Gamma_L^{\alpha/2} (\mathbb{H}^n)} +C\|f\|_{\Gamma_L^{\alpha/2} (\mathbb{H}^n)} \\
 & \leq C\|f\|_{\Gamma_L^{\alpha/2} (\mathbb{H}^n)}.
\end{align*}
The latter estimate together with \eqref{eq:1+x finite}
implies that $f \in \widetilde{\Gamma}^{\alpha/2} (\mathbb{H}^n)$ and 
\begin{align*}
 \|f\|_{\widetilde{\Gamma}^{\alpha/2} (\mathbb{H}^n)}    \leq C\|f\|_{\Gamma^{\alpha/2}_L (\mathbb{H}^n)}.
\end{align*}
Then, by Lemma  \ref{characterizatoin of Lip heat}, we have 
\begin{align*}
  \sup_{|\y| >0}
      \frac{\|f(\x \y) + f(\x \y^{-1}) -2 f(\x)\|_{L^\infty(\mathbb{H}^n_{\x})}}{|\y|^\alpha}\leq C \|f\|_{\widetilde{\Gamma}^{\alpha/2} (\mathbb{H}^n)}    \leq C\|f\|_{\Gamma^{\alpha/2}_L (\mathbb{H}^n)}.    
\end{align*}
This estimate coupled with \eqref{eq:rho alpha finite} implies that $f \in \Lambda_L^\alpha (\mathbb{H}^n)$ and
\begin{align*}
 \|f\|_{\Lambda^\alpha_L (\mathbb{H}^n)}   \leq C\|f\|_{\Gamma^{\alpha/2}_L (\mathbb{H}^n)}.
\end{align*}
Hence $\Gamma_L^{\alpha/2} (\mathbb{H}^n) \subset \Lambda_L^\alpha (\mathbb{H}^n)$ and the inclusion
map is continuous.
\end{proof}

\section{Regularity of fractional powers of $L=-\Delta_{\mathbb{H}^n} +V$} \label{sec:regularity}
Following \cite{ST0}, we introduce the Bessel potentials, fractional integrals and fractional ``sub-Laplacian''
associated to the operator $L =-\Delta_{\mathbb{H}^n}+V$
on $\mathbb{H}^n$ as follows.

\begin{itemize}
    \item  {\it The Bessel potential of order $\beta>0$},
    \begin{align*}
        (Id + L)^{-\beta /2} f(\x)= \frac{1}{\Gamma(\beta/2)} \int_0^\infty e^{-t}e^{-tL}f(\x) t^{\beta /2} \frac{dt}{t}.
    \end{align*}
    \item {\it The fractional integral of order $\beta >0$},
    \begin{align*}
      L^{-\beta /2} f(\x)= \frac{1}{\Gamma(\beta/2)} \int_0^\infty  e^{-tL}f(\x) t^{\beta /2} \frac{dt}{t}.
    \end{align*}
    \item {\it The fractional ``sub-Laplacian'' of order $\beta/2 >0$},
        \begin{align*}
      L^{\beta /2} f(\x)= \frac{1}{c_\beta} \int_0^\infty (Id-  e^{-tL})^{\lfloor \beta/2 \rfloor +1}f(\x)
       \frac{dt}{t^{1 + (\beta/2)}}.
    \end{align*}
\end{itemize}

Arguing similarly as in \cite[Theorems 1.6 and 1.7]{DT}, one can prove the following  results. We omit their proofs.
\begin{theorem}
Let $\alpha,\beta >0$ and $\mathcal{T}_\beta$ denote $(Id+ L)^{-\beta/2}$ or $L^{-\beta /2}$. Then
we have
\begin{enumerate}[\rm (i)]
    \item $\|\mathcal{T}_\beta f\|_{\Gamma_L^{(\alpha +\beta)/2}(\mathbb{H}^n)} \leq C \|f\|_{\Gamma_L^{\alpha/2}(\mathbb{H}^n)}$. 
    \item $\|\mathcal{T}_\beta f\|_{\Gamma_L^{\beta/2}(\mathbb{H}^n)} \leq C \|f\|_{L^\infty (\mathbb{H}^n)}$. 
\end{enumerate}
\end{theorem}
\begin{theorem}
 Let $0<\beta <\alpha$ and $f \in \Gamma^{\alpha /2}_L (\mathbb{H}^n)$. Then
 \begin{align*}
 \|\mathcal{T}_\beta f\|_{\Gamma_L^{(\alpha -\beta)/2}(\mathbb{H}^n)} \leq C \|f\|_{\Gamma_L^{\alpha/2}(\mathbb{H}^n)}.    
 \end{align*}
\end{theorem}

\end{document}